\documentclass[12pt]{article}
\usepackage[utf8]{inputenc}
\usepackage{amsfonts}
\usepackage{scrextend}
\usepackage{mathtools}
\usepackage{float}
\usepackage{amsmath}
\usepackage[margin=1.2in]{geometry}
\usepackage{authblk}
\usepackage{amssymb}
\usepackage{epigraph}
\usepackage[english]{babel}
\usepackage[nottoc]{tocbibind}
\usepackage{tikz-cd}
\usepackage{graphicx}
\usetikzlibrary{matrix}
\usepackage{amsthm}
\usepackage[mathscr]{euscript}
 \let\mathscr\relax
\usepackage[scr]{rsfso}
\usepackage{comment}

\usepackage[multiple]{footmisc}

\usepackage{bm, mathdots, accents}

\newtheorem{nummer}{ }[section]

\newtheorem*{mainthm}{\sc Main Theorem}
\newtheorem{prp}[nummer]{\sc Proposition}
\newtheorem{lem}[nummer]{\sc Lemma}

\newtheorem{fct}[nummer]{\sc Fact}

\newtheorem{defi}[nummer]{\sc Definition}

\newenvironment{claim}[1]{\par\noindent\textsc{Claim.}\space#1}{}

\newcounter{faelle} 
\renewcommand{\thefaelle}{\rm(\alph{faelle})}

\renewcommand\qed{\relax\ifmmode~\hfill$\dashv$\else\unskip\nobreak~\hfill$\dashv$\fi}

\def\epsilon{\varepsilon}

\newcommand\func{{}^\omega\omega}

\newcommand\cov{\mathit{cov}}

\renewcommand{\phi}{\varphi}
\renewcommand{\theta}{\vartheta}

\begin{document}

\begin{center}
{\Large\sc A unique $Q$-point and infinitely many near-coherence classes of ultrafilters}\\[1.8ex]

{\small Lorenz Halbeisen}\\[1.2ex] 
{\scriptsize Department of Mathematics, ETH Z\"urich, 
8092 Z\"urich, Switzerland\\ 
lorenz.halbeisen@math.ethz.ch}\\[1.8ex]

{\small Silvan Horvath}\\[1.2ex] 
{\scriptsize Department of Mathematics, ETH Z\"urich, 
8092 Z\"urich, Switzerland\\ 
silvan.horvath@math.ethz.ch}\\[1.8ex]

{\small Saharon Shelah
\footnote{
Research partially supported by the {\it Israel 
Science Foundation\/} grant no.\;{2320/23}.
This is paper\;{1265} on the author's publication list.
}}\\[1.2ex]
{\scriptsize Einstein Institute of Mathematics,
The Hebrew University of Jerusalem,
9190401 Jerusalem, Israel\\
shelah@math.huji.ac.il}\\[1.2ex]
{\scriptsize\sl and}\\[1ex]
{\scriptsize Department of Mathematics,
Hill Center\,--\,Busch Campus, Rutgers, 
State University of New Jersey\\
110 Frelinghuysen Rd.,
Piscataway, NJ 08854-8019, U.S.A.
}
\end{center}


\begin{quote}
{\small {\bf Abstract.} 
We show that in the model obtained by iteratively pseudo-intersecting a Ramsey ultrafilter via a length-$\omega_2$ countable support iteration of restricted Mathias forcing over a ground model satisfying $\textsf{CH}$, there is a unique $Q$-point up to isomorphism. In particular, it is consistent that there is only one $Q$-point while there are $2^{\mathfrak{c}}$-many near-coherence classes of ultrafilters.}
\end{quote}

\begin{quote}
\small{{\bf key-words\/}: $Q$-point,  Ramsey ultrafilter, Mathias forcing}\\
\small{\bf 2010 Mathematics Subject Classification\/}: {\bf 03E35}\ {03E17}
\end{quote}


\setcounter{section}{-1}
\section{Introduction}

Throughout this paper, read \textit{ultrafilter} as \textit{non-principal ultrafilter on $\omega$}. For $x \subseteq \omega$, we denote by $[x]^\omega$ the set of infinite subsets of $x$ and by $[x]^{<\omega}$ the set of finite subsets of $x$.

Recall that an ultrafilter $E$ is a \textit{$Q$-point} if and only if for every interval partition $\{[k_i, k_{i+1}): i \in \omega\}$ of $\omega$, there exists some $x \in E$ such that $\forall i \in \omega: |x \cap [k_i, k_{i+1})| \leq 1$. Furthermore, an ultrafilter $\mathcal{U}$ is a \textit{Ramsey ultrafilter} if and only if the \textit{Maiden} has no winning strategy in the \textit{ultrafilter game for $\mathcal{U}$}, played between the Maiden and \textit{Death}:

\begin{defi}
Let $\mathcal{U}$ be an ultrafilter. The ultrafilter game for $\mathcal{U}$ proceeds as follows:

The Maiden opens the game and plays some $y_0 \in \mathcal{U}$. Death responds by playing some $n_0 \in y_0$. In the $(k+1)$-th move, the Maiden having played $y_0 \supseteq y_1 \supseteq ... \supseteq y_k$, and Death having played $n_0 < n_1 < ... < n_k$, the Maiden plays some $y_{k+1}\in [y_k]^\omega \cap \mathcal{U}$, and Death responds by playing some $n_{k+1} \in y_{k+1}$, $n_{k+1}>n_k$.

Death wins if and only if $\{n_i: i \in \omega\}\in \mathcal{U}$.
\end{defi}

It is well-known that every Ramsey ultrafilter is a $Q$-point. Canjar~\cite{canjar1990generic} showed that the existence of $2^{\mathfrak{c}}$-many Ramsey ultrafilters follows from the assumption $\cov(\mathcal{M})=\mathfrak{c}$. The weaker assumption $\cov(\mathcal{M})=\mathfrak{d}$ implies the existence of $2^{\mathfrak{c}}$ $Q$-points, as was shown by Mill{\'a}n~\cite{millan2007note}. It is well-known that in the Mathias model -- the model obtained by a length-$\omega_2$ countable support iteration of unrestricted Mathias forcing over a ground model satisfying \textsf{CH} -- there are no $Q$-points (see~\cite[Proposition 26.23]{cst}). In fact, the Mathias model contains no \textit{rapid} ultrafilters, where an ultrafilter $E$ is rapid if for every $f\in \func$ there exists some $x\in E$ such that $\forall n \in \omega: |x \cap f(n)|\leq n$ (note that every $Q$-point is rapid). It follows that both the Mathias model and the model considered in this paper satisfy $\cov(\mathcal{M})=\omega_1 < \mathfrak{d}=\mathfrak{c}=\omega_2$. 

In contrast to the Mathias model, our model contains $2^{\mathfrak{c}}$-many rapid ultrafilters: It follows from an observation of Mill{\'a}n~\cite[page 222]{millan2007note} that the existence of a single rapid ultrafilter $E$ implies the existence of $2^{\mathfrak{c}}$ of them, by considering the products $\mathcal{U} \times E$ for different ultrafilters $\mathcal{U}$.\footnote{$\mathcal{U} \times E$ is an ultrafilter on $\omega \times \omega$ defined by $\mathcal{U} \times E = \{x \subseteq \omega \times \omega: \{n \in \omega: \{m \in \omega: \langle n,m\rangle \in x\} \in E\}\in \mathcal{U}\}$.}

While the consistency of the non-existence of $Q$-points is a well-established fact with a variety of witnesses apart from the Mathias model\footnote{such as the Laver and Miller models (see~\cite{miller1980there} and~\cite{blass1989near}, respectively).}, the construction of models containing only `few' $Q$-points seems to have received less attention. However, such models do arise naturally as models containing only few near-coherence classes of ultrafilters\footnote{Two ultrafilters $\mathcal{U}_1$ and $\mathcal{U}_2$ are \textit{nearly-coherent} if there is some finite-to-one $f\in {^\omega}{\omega}$ such that $f(\mathcal{U}_1)=f(\mathcal{U}_2)$, where $f(\mathcal{U}_i):=\{X \subseteq \omega: f^{-1}[X] \in \mathcal{U}_i\}$. Note that two $Q$-points are nearly-coherent if and only if they are isomorphic.}: Indeed, Mildenberger~\cite{mildenberger2024exactly} has constructed models with exactly two and exactly three near-coherence classes, and it is easy to see that these contain exactly one and exactly two $Q$-points, respectively: In her model with exactly two near-coherence classes, one class contains a Ramsey ultrafilter, while the other class contains an ultrafilter that is $\omega_1$-generated. Hence, this latter class cannot contain a $Q$-point, since her models satisfy $\mathfrak{d}=\mathfrak{c}=\omega_2$ and such a $Q$-point would thus have to be ${<}\mathfrak{d}$-generated, which is impossible. Analogously, in  Mildenberger's model with exactly three near-coherence classes, two classes are represented by Ramsey ultrafilters, while the third contains an $\omega_1$-generated ultrafilter -- giving exactly two $Q$-points in total.

The construction of models with exactly $n$ near-coherence classes of ultrafilters for various finite $n \geq 4$ would similarly yield the consistency of exactly $m$ $Q$-points for some $m < n$.\footnote{The inequality is strict since such a model must satisfy $\mathfrak{u}<\mathfrak{d}$, a result due to Banakh and Blass~\cite{banakh2006number}. Hence, one of the $n$ near-coherence classes contains a ${<}\mathfrak{d}$-generated ultrafilter and thus no $Q$-point.}\footnote{See the note on the next page.}

\pagebreak

The model considered in this paper is of a different nature, however: It contains only one $Q$-point while its number of near-coherence classes is $2^{\mathfrak{c}}$, i.e., the model's lack of $Q$-points is not the consequence of a lack of near-coherence classes. This follows from the fact that dominating reals are added at each of the $\omega_2$ stages of the iteration, which gives $\mathfrak{b}=\mathfrak{d}=\mathfrak{c}=\omega_2$ in the final extension. Since $\mathfrak{b} \leq \mathfrak{u}$ (see Solomon~\cite{solomon1977families}), we have $\mathfrak{u}=\mathfrak{d}=\omega_2$ in our model, and hence there are $2^{\mathfrak{c}}$-many near-coherence classes of ultrafilters by Banakh and Blass~\cite{banakh2006number}.

Incidentally, both our model as well as Mildenberger's models answer the two questions posed in a recent paper of Raghavan \cite{raghavan2025}. He asked whether it is consistent that there are $Q$-points, while (1) there is no Tukey maximal $Q$-point, and (2) the $Q$-points are not cofinal in the RK-ordering.\footnote{See, for example, \cite{raghavan2012cofinal} and \cite{dobrinen2011tukey} on the Tukey- and RK-orderings of ultrafilters.} Note that the unique $Q$-point in our model is simultaneously a $P$-point, and these are never Tukey maximal (see \cite[Corollary 19]{dobrinen2011tukey}).

\paragraph{Note added in proof:} 
Based on the technique used below, in the forthcoming paper~\cite{halbeisen2025there},
Halbeisen, Horvath, and \"Ozalp have established the consistency of the
statement ``{\sl There are exactly $n$ $Q$-points up to isomorphism and $\mathfrak{u}=\mathfrak{d}=\mathfrak{c}$''}, for any
$n\in\omega$.

\section{Definitions and Preliminaries}

Before we state and prove our main result in the next section, let
us give some definitions, relations, and related results.

\begin{defi}
Let $\mathcal{U}$ be a Ramsey ultrafilter. Mathias forcing restricted to $\mathcal{U}$, written $\mathbb{M}_{\mathcal{U}}$, consists of conditions $\langle s, x\rangle \in [\omega]^{<\omega}\times \mathcal{U}$ with $\max s < \min x$, ordered by
\[\langle s,x\rangle \leq_{\mathbb{M}_{\mathcal{U}}} \langle t, y\rangle :\iff s \supseteq t \land x \subseteq y \land s \setminus t \subseteq y.\]
\end{defi}
The forcing notion $\mathbb{M}_{\mathcal{U}}$ clearly satisfies the c.c.c. and is therefore proper. We will need the following additional facts.

\begin{fct}[{e.g., see~\cite[Theorem 26.3]{cst}}]
Let $\mathcal{U}$ be a Ramsey ultrafilter. The forcing notion $\mathbb{M}_{\mathcal{U}}$ has the pure decision property, i.e., for any sentence $\phi$ in the forcing language and any $\mathbb{M}_{\mathcal{U}}$-condition $\langle s, x \rangle$, there exists $y \in [x]^{\omega} \cap \mathcal{U}$ such that either $\langle s,y\rangle \Vdash_{\mathbb{M}_{\mathcal{U}}} \phi$ or $\langle s,y\rangle \Vdash_{\mathbb{M}_{\mathcal{U}}} \neg \phi$.
\end{fct}

\begin{defi}
	Recall that a forcing notion $\mathbb{P}$ has the Laver property if for every $\mathbb{P}$-name $\undertilde{g}$ for an element of ${^\omega}\omega$ such that there exists $f\in {^\omega}\omega \cap \mathbf{V}$ with 
\[\mathbb{P}\Vdash \forall n \in \omega: \undertilde{g}(n) \leq f(n),\]
we have that $\mathbb{P}$ forces that there exists $c: \omega \to [\omega]^{<\omega}$ in $\mathbf{V}$ with
\[\forall n \in \omega: |c(n)| \leq 2^{n} \text{ and } \undertilde{g}(n) \in c(n).\]
\end{defi}

\begin{fct}[{e.g., see~\cite[Corollary 26.8]{cst}}]\label{fct:mathiaslaver}
Let $\mathcal{U}$ be a Ramsey ultrafilter. The forcing notion $\mathbb{M}_{\mathcal{U}}$ has the Laver property.
\end{fct}

\begin{fct}[{e.g., see~\cite[Ch. VI, 2.10D]{proper}}]\label{fct:preservationlaver}
	The Laver property is preserved under countable support iterations of proper forcing notions.
\end{fct}

\section{Result}

Now, we are ready to state the 

\begin{mainthm}
It is consistent that there is a unique $Q$-point while there are  $2^{\mathfrak{c}}$-many near-coherence classes of ultrafilters.
\end{mainthm}

\rm The proof is given by the following construction and the subsequent
results:
Assume that the ground model $\mathbf{V}$ satisfies $\textsf{CH}$. 
By induction, we define:
\begin{itemize}
\item[(i)] A countable support iteration $\mathbb{P}_{\omega_2}:=\langle \mathbb{P}_{\xi}, \undertilde{Q}_\xi: \xi \in \omega_2\rangle$ of c.c.c. forcing notions,
\item[(ii)] A sequence $\langle \undertilde{\mathcal{U}}_\xi: \xi \in \omega_2\rangle$, such that
\[\forall \xi \in \omega_2: \mathbb{P}_\xi \Vdash ``\undertilde{\mathcal{U}}_\xi \text{ is a Ramsey ultrafilter extending } \bigcup_{\iota \in \xi} \undertilde{\mathcal{U}}_\iota"\]
and $\undertilde{Q}_\xi$ is a $\mathbb{P}_\xi$-name for Mathias forcing restricted to $\undertilde{\mathcal{U}}_\xi$,
\end{itemize} 

Assume that we are in step $\xi \in \omega_2$. Let $G_\xi$ be $\mathbb{P}_\xi$-generic over $\mathbf{V}$  and work in $\mathbf{V}[G_\xi]$. Note that since $\mathbb{P}_\xi$ is a countable support iteration of proper forcing notions that are forced to be of size $\leq \omega_1$, we have $\mathbf{V}[G_\xi]\models \textsf{CH}$ (e.g., see~\cite[Theorem 2.12]{abraham2009proper}). For each $\iota \in \xi$, let $\eta_{\iota}$ be the Mathias real added at stage $\iota$.

If $\xi=\xi'+1$, $\eta_{\xi'}$ pseudo-intersects $\undertilde{\mathcal{U}}_{\xi'}[G_\xi]$ and we may construct a Ramsey ultrafilter on $\eta_{\xi'}$ using \textsf{CH} (and extend it to $\omega$ to obtain $\mathcal{U}_\xi$). Similarly, if $\xi$ is a limit ordinal and $\text{cf}(\xi)=\omega$, we can build $\mathcal{U}_\xi$ on a pseudo-intersection of the tower $\langle \eta_\iota: \iota \in \xi\rangle$. Finally, if $\text{cf}(\xi)=\omega_1$, then $\bigcup_{\iota \in \xi} \undertilde{\mathcal{U}}_\iota[G_\xi]$ is already a Ramsey ultrafilter, since no new reals are added at stage $\xi$. For the same reason we also have that $\mathcal{U}_{\omega_2}:=\bigcup_{\xi \in \omega_2} \undertilde{\mathcal{U}}_\xi[G]$ is a Ramsey ultrafilter in $\mathbf{V}[G]$, where $G$ is $\mathbb{P}_{\omega_2}$-generic over $\mathbf{V}$.

\begin{fct}[see, e.g.,{~\cite[Theorem 2.10]{abraham2009proper}}]
$\mathbb{P}_{\omega_2}$ is proper and satisfies $\omega_2$-c.c..
\end{fct}

We need to show that $\mathcal{U}_{\omega_2}$ is the only $Q$-point in $\mathbf{V}[G]$. To see this, assume by contradiction that $\mathbf{V}[G]\models ``E\text{ is a }Q\text{-point and not isomorphic to }\mathcal{U}_{\omega_2}"$.

\begin{lem}
There exists $\delta \in \omega_2$ such that $E \cap \mathbf{V}[G_\delta] \in \mathbf{V}[G_\delta]$ and $\mathbf{V}[G_\delta] \models ``E \cap \mathbf{V}[G_\delta] \text{ is a }Q\text{-point and not isomorphic to }\mathcal{U}_\delta"$.
\end{lem}

\begin{proof}

 Fix $\xi \in \omega_2$ and consider names $\undertilde{e}_\xi$, $\undertilde{i}_\xi$, $\undertilde{s}_\xi$, $\undertilde{b}_\xi$ and $\undertilde{f}_\xi$ such that $\mathbb{P}_{\omega_2}$ forces that
 \begin{itemize}
 \item[(i)] $``\undertilde{e}_\xi$ is an enumeration (in $\omega_1$) of $\undertilde{E} \cap \mathbf{V}[\underaccent{\dot}{G}_\xi]"$. For each $\alpha \in \omega_1$ and $n \in \omega$ let $\mathcal{E}_{\xi, \alpha, n} \subseteq \mathbb{P}_{\omega_2}$ be a maximal antichain deciding $``n \in \undertilde{e}_\xi(\alpha)"$.
 \item[(ii)] $``\undertilde{i}_\xi$ is an enumeration (in $\omega_1$) of the set of interval partitions of $\omega$ in $\mathbf{V}[\underaccent{\dot}{G}_\xi]"$. Note that we may assume that $\undertilde{i}_\xi$ is a $\mathbb{P}_\xi$-name.
 \item[(iii)] ``For all $\alpha \in \omega_1$, $\undertilde{s}_\xi(\alpha)$ is an element of $\undertilde{E}$ that intersects each interval in the interval partition $\undertilde{i}_{\xi}(\alpha)$ in at most one point". Let $\mathcal{S}_{\xi, \alpha, n}\subseteq \mathbb{P}_{\omega_2}$ be a maximal antichain deciding $``n \in \undertilde{s}_\xi(\alpha)"$.
 \item[(iv)] $``\undertilde{b}_\xi$ is an enumeration (in $\omega_1$) of all permutations of $\omega$ in $\mathbf{V}[\underaccent{\dot}{G}_\xi]"$. We may again assume that $\undertilde{b}_\xi$ is a $\mathbb{P}_\xi$-name.
 \item[(v)] ``For all $\alpha \in \omega_1$, $\undertilde{f}_\xi(\alpha)$ is a pair $\langle \undertilde{x}_\alpha, \undertilde{y}_\alpha \rangle$ such that $\undertilde{x}_\alpha$ is in $\undertilde{E}$,  $\undertilde{y}_\alpha$ is in $\undertilde{\mathcal{U}}_{\omega_2}$ and $\undertilde{b}_\xi(\alpha)[\undertilde{x}_\alpha]$ is disjoint from $\undertilde{y}_\alpha"$. Let $\mathcal{X}_{\xi, \alpha, n} \subseteq \mathbb{P}_{\omega_2}$ be a maximal antichain deciding $``n \in \undertilde{x}_\alpha"$, and define $\mathcal{Y}_{\xi, \alpha,n}$ analogously.
 \end{itemize}
 
Since $\mathbb{P}_{\omega_2}$ satisfies $\omega_2$-c.c., 
there exists for each $\xi \in \omega_2$ some $\gamma_\xi \in \omega_2$ greater than $\xi$ such that all the above antichains consist of $\mathbb{P}_{\gamma_\xi}$-conditions. Recursively define $\lambda(0)=0$, $\lambda(\xi+1)=\gamma_{\lambda(\xi)}$ and for limit ordinals $\xi: \lambda(\xi)=\bigcup_{\iota \in \xi}\lambda(\iota)$, for $\xi \leq \omega_1$. Set $\delta:=\lambda(\omega_1)$ and consider the extension $\mathbf{V}[G_\delta]$. Since $\text{cf}(\delta)=\omega_1$, we have that $E \cap \mathbf{V}[G_\delta] =\bigcup_{\iota \in \omega_1}E \cap \mathbf{V}[G_{\lambda(\iota)}] $, and since each $E \cap \mathbf{V}[G_{\lambda(\iota)}]$ is an element of $\mathbf{V}[G_\delta]$ by (i), $E \cap \mathbf{V}[G_\delta]$ is an element of $\mathbf{V}[G_\delta]$ (and an ultrafilter). Furthermore, any interval partition of $\omega$ in $\mathbf{V}[G_\delta]$ already appears in some $\mathbf{V}[G_{\lambda(\iota)}]$, $\iota \in  \omega_1$, where it equals $\undertilde{i}_{\lambda(\iota)}[G_ {\lambda(\iota)}](\alpha)$ for some $\alpha \in \omega_1$. Since $\undertilde{s}_{\lambda(\iota)}[G_\delta](\alpha) \in E \cap \mathbf{V}[G_\delta]$, we obtain that $E \cap \mathbf{V}[G_\delta]$ is a $Q$-point. Finally and analogously, any permutation of $\omega$ in $\mathbf{V}[G_\delta]$ already appears in $\mathbf{V}[G_ {\lambda(\iota)}]$ for some $\iota \in \omega_1$ and hence there are witnesses $\undertilde{x}_\alpha[G_\delta] \in E \cap \mathbf{V}[G_\delta]$ and $\undertilde{y}_\alpha[G_\delta] \in \mathcal{U}_{\omega_2} \cap \mathbf{V}[G_\delta]=\mathcal{U}_\delta$ witnessing that $E \cap \mathbf{V}[G_\delta]$ and $\mathcal{U}_\delta$ are not isomorphic.
\end{proof}

We now designate $\mathbf{V}[G_\delta]$ as the new ground model and rename the $Q$-point $E \cap \mathbf{V}[G_\delta]$ to $E$ and the Ramsey ultrafilter $\mathcal{U}_{\delta}$ to $\mathcal{U}$. Note that by the Factor-Lemma (e.g., see~\cite[Theorem 4.6]{goldstern1992tools}), the quotient $\mathbb{P}_{\omega_2}/G_{\delta}$ is again isomorphic to a countable support iteration of restricted Mathias forcings. In particular, by Facts \ref{fct:mathiaslaver} and \ref{fct:preservationlaver}, $\mathbb{P}_{\omega_2}/G_{\delta}$ is isomorphic to the two-step iteration $\mathbb{M}_{\mathcal{U}}\ast \undertilde{R}$, where $\mathbb{M}_{\mathcal{U}} \Vdash ``\undertilde{R} \text{ has the Laver property}"$.

It remains to show the following

\begin{prp}
	Let $E$ be a $Q$-point and $\mathcal{U}$ a Ramsey ultrafilter such that $E$ and $\mathcal{U}$ are not isomorphic. Let $\mathbb{M}_{\mathcal{U}}$ be Mathias forcing restricted to $\mathcal{U}$ and let $\undertilde{R}$ be a $\mathbb{M}_{\mathcal{U}}$-name such that $\mathbb{M}_{\mathcal{U}} \Vdash ``\undertilde{R}$ has the Laver property". Then $\mathbb{M}_{\mathcal{U}} \ast \undertilde{R}\Vdash ``E$ cannot be extended to a $Q$-point".
\end{prp}

\begin{proof}
It suffices to show that if $\langle p, \undertilde{q}\rangle \in \mathbb{M}_{\mathcal{U}} \ast \undertilde{R}$ and a $\mathbb{M}_{\mathcal{U}} \ast \undertilde{R}$-name $\undertilde{a}$ for a strictly increasing element of ${^\omega}\omega$ are such that
\[\langle p, \undertilde{q}\rangle \Vdash_{\mathbb{M}_{\mathcal{U}} \ast \undertilde{R}} \forall n \in \omega: \undertilde{a}(n) \in (\underaccent{\dot}{\eta}(n-1), \underaccent{\dot}{\eta}(n)],\]
then there exists some $v \in E$ and some $\langle \bar p, \undertilde{\bar q} \rangle \leq_{\mathbb{M}_{\mathcal{U}} \ast \undertilde{R}} \langle p, \undertilde{q}\rangle $ such that
\[\langle \bar p, \undertilde{\bar q}\rangle \Vdash_{\mathbb{M}_{\mathcal{U}} \ast \undertilde{R}} |\text{range}(\undertilde{a}) \cap v| < \omega.\]
Recall that $\underaccent{\dot}{\eta}$ is the canonical $\mathbb{M}_{\mathcal{U}}$-name for the Mathias real (assume $\mathbb{M}_{\mathcal{U}}\Vdash \underaccent{\dot}{\eta}(-1)=-\infty$).

Note that $\undertilde{a}$ is forced by $\mathbb{M}_{\mathcal{U}}$ to be dominated by $\underaccent{\dot}{\eta}$. Hence, by the Laver property of $\undertilde{R}$, there exists a $\mathbb{M}_{\mathcal{U}}$-name $\undertilde{c}$ for a function from $\omega$ to $[\omega]^{<\omega}$ and some $\langle p', \undertilde{q}'\rangle \leq_{\mathbb{M}_{\mathcal{U}} \ast \undertilde{R}} \langle p, \undertilde{q}\rangle$ such that
\[\langle p', \undertilde{q}'\rangle \Vdash_{\mathbb{M}_{\mathcal{U}} \ast \undertilde{R}} \forall n \in \omega: \undertilde{a}(n) \in \undertilde{c}(n) \text{ and } |\undertilde{c}(n)|\leq 2^n.\]

We may assume without loss of generality that $p' \Vdash_{\mathbb{M}_{\mathcal{U}}} \forall n \in \omega: \undertilde{c}(n)\subseteq (\underaccent{\dot}{\eta}(n-1), \underaccent{\dot}{\eta}(n)]$. Let $\undertilde{C}$ be a $\mathbb{M}_{\mathcal{U}}$-name for an element of $[\omega]^\omega$ such that $p' \Vdash_{\mathbb{M}_{\mathcal{U}}} \undertilde{C}= \bigcup\text{range}(\undertilde{c})$. Hence, we have
\[\langle p', \undertilde{q}'\rangle \Vdash_{\mathbb{M}_{\mathcal{U}} \ast \undertilde{R}} \forall n \in \omega: \undertilde{a}(n) \in \undertilde{C} \cap (\underaccent{\dot}{\eta}(n-1), \underaccent{\dot}{\eta}(n)] \text{ and } | \undertilde{C} \cap (\underaccent{\dot}{\eta}(n-1), \underaccent{\dot}{\eta}(n)]|\leq 2^n.\]

\begin{lem}\label{lemma1}
	Write $p'=\langle s, x_0\rangle$. There exists $x_1 \in [x_0]^{\omega}\cap \mathcal{U}$ such that the $\mathbb{M}_{\mathcal{U}}$-condition $\langle s, x_1\rangle \leq_{\mathbb{M}_{\mathcal{U}}} \langle s, x_0\rangle$ has the following property:
	
	For every $t \in [x_1]^{<\omega}$, there exists $C_t \in [\omega]^{<\omega}$ such that
	\[\langle s \cup t, x_1\setminus (\max t)^{+}\rangle \Vdash_{\mathbb{M}_{\mathcal{U}}} \undertilde{C}\cap (\max t)^{+} =C_t.\]
\end{lem}

\begin{proof}
We define a strategy for the Maiden in the ultrafilter game for $\mathcal{U}$, which will not be a winning strategy since $\mathcal{U}$ is a Ramsey ultrafilter. 

Since $\mathbb{M}_{\mathcal{U}}$ has pure decision, there exists $C_{\emptyset} \subseteq (\max s)^{+}$ and $y_0 \in [x_0]^{\omega}\cap \mathcal{U}$ such that $\langle s, y_0\rangle \Vdash_{\mathbb{M}_{\mathcal{U}}} \undertilde{C} \cap (\max s)^{+} = C_{\emptyset}$. The Maiden starts by playing $y_0$.

Assume $y_0 \supseteq y_1 \supseteq ... \supseteq y_k$ and $n_0 < n_1 < ... < n_k$ have been played, where $\forall i \leq k: y_i \in \mathcal{U}$ and $n_i \in y_i$. Again by pure decision, for each $t \subseteq \{n_0, n_1,..., n_k\}$ with $\max t=n_k$, there exists $z_t \in [y_k \setminus n_k^{+}]^{\omega}\cap \mathcal{U}$ and $C_t \subseteq n_k^{+}$ such that $\langle s \cup t, z_t\rangle \Vdash_{\mathbb{M}_{\mathcal{U}}} \undertilde{C}\cap (n_k)^{+}=C_t$. The Maiden plays
\[y_{k+1}:=\bigcap_{\substack{{t \subseteq \{n_i: i \leq k\}}\\{\max t =n_k}}} z_t.\]

Since Death wins, we have that $x_1:=\{n_i: i \in \omega\} \in \mathcal{U}$. It is easy to check that this $x_1$ satisfies the lemma.
\end{proof}

The following lemma strengthens the previous one.

\begin{lem}\label{lemma2}
Assume $\langle s, x_1\rangle$ is as in the conclusion of the previous lemma. There exists $x_2 \in [x_1]^{\omega}\cap \mathcal{U}$ such that $\langle s, x_2\rangle$ has the following property:

	For every $t \in [x_2]^{<\omega}$, every $m \in x_2\setminus \max t$ and all $n, n' \in x_2 \setminus m^{+}$, it holds that $C_{t \cup \{n\}} \cap m^{+} = C_{t \cup \{n'\}} \cap m^{+}$.
\end{lem}

\begin{proof}
We again prove this by playing the ultrafilter game for $\mathcal{U}$. Assume $y_0:=x_1 \supseteq y_1 \supseteq ... \supseteq y_k$ and $n_0 < n_1 < ... < n_k$ have been played. For every $t \subseteq \{n_0, n_1, ..., n_k\}$ and every $d \subseteq n_k^{+}$ consider the set 
\[P_{t,d}:=\{n \in y_k \setminus n_k^{+}: C_{t \cup \{n\}} \cap n_k^{+}=d\}.\]

Note that for every $t \subseteq \{n_0, n_1, ..., n_k\}$, the set $\{P_{t,d}: d \subseteq n_k^{+}\}$ is a partition of $y_k \setminus n_k^{+}$ into finitely many pieces. Hence, there exists one $d_t\subseteq n_k^{+}$ such that $P_{t,d_t}\in \mathcal{U}$.

The Maiden plays 
\[y_{k+1}:=\bigcap_{t \subseteq \{n_i: i \leq k\}} P_{t,d_t}.\]

Death will win and hence $x_2:=\{n_i: i \in \omega\} \in \mathcal{U}$. It is again not hard to check that $x_2$ satisfies the lemma.
\end{proof}

The following fact will be needed later.

\begin{fct}\label{speedfact}
Without loss of generality, we may assume that for all $n \in \{\max s\} \cup x_2$, if $n$ is the $j$'th element of $s \cup x_2$ in increasing order, then $n > 2^{j+1}$.
\end{fct}

\begin{proof}
Note that the conclusion of Lemmata \ref{lemma1} and \ref{lemma2} also holds for each $\langle s', x' \rangle \leq_{\mathbb{M}_{\mathcal{U}}} \langle s, x_2\rangle$. Hence, we simply trim $x_2$ such that the enumeration of $s \cup x_2$ dominates $2^{j+1}$ above $|s|$ and replace $s$ with $s \cup \{\min x_2\}$ and $x_2$ with $x_2 \setminus \{\min x_2\}$.
\end{proof}

Next, let $N$ be a countable elementary submodel of some large enough $\mathcal{H}_{\chi}$ such that $\{\mathcal{U}, \mathbb{M}_{\mathcal{U}}, \undertilde{C}, \langle s, x_2\rangle\} \in N$. By induction, construct a sequence $N_0 \subseteq N_1 \subseteq ... $ of finite subsets of $N$ such that

\begin{itemize}
\item[(i)] $\{\mathcal{U}, \mathbb{M}_{\mathcal{U}}, \undertilde{C}, \langle s, x_2\rangle, s, x_2\} \subseteq N_0$,
\item[(ii)] $\bigcup_{i \in \omega}N_i =N	$,
\item[(iii)] $\forall i \in \omega: k_i:=N_i \cap \omega \in \omega$.
\item[(iv)] $\forall i \in \omega: \forall t \in [\omega]^{<\omega}: t \in N_i \iff t \subseteq N_i$,
\item[(v)] If $\langle m,l,D\rangle \in (\omega \times \omega \times [\omega]^{<\omega})\cap N_i$, then $m,l,D \in N_i$ (and hence $D \subseteq N_i$ by the previous condition).
\item[(vi)] $\forall i \in \omega:$ If $\phi(x, a_0, ...,a_l)$ is a formula of length less than $2025$ with $a_0, ..., a_l \in N_i$ and $N \models \exists x \phi(x, a_0, ...,a_l)$, then there exists $b \in N_{i+1}$ such that $N \models \phi(b, a_0, ..., a_l)$.
\end{itemize}

\begin{lem}\label{lemma3}
	$\langle s, x_2\rangle$ forces that
	\[\forall i \in \omega\setminus \{0,1\}: \undertilde{C} \setminus (\max s)^{+}\cap [k_{i-1}, k_{i})\neq \emptyset \implies
	\begin{cases}
	\text{range}(\underaccent{\dot}{\eta})\cap [k_{i-2}, k_{i-1})\neq \emptyset, \text{ or}	\\
	\text{range}(\underaccent{\dot}{\eta})\cap [k_{i-1}, k_i)\neq \emptyset, \text{ or}	\\
	\text{range}(\underaccent{\dot}{\eta})\cap [k_{i}, k_{i+1})\neq \emptyset.
	\end{cases}
\]
\end{lem}

\begin{proof}
Assume $\langle s \cup t, x'\rangle\leq_{\mathbb{M}_{\mathcal{U}}} \langle s, x_2\rangle$, $a \in \omega \setminus (\max s)^{+}$ and $i \in \omega \setminus \{0,1\}$ are such that
\[\langle s \cup t, x'\rangle \Vdash_{\mathbb{M}_{\mathcal{U}}} a \in \undertilde{C} \setminus (\max s)^{+}\cap [k_{i-1}, k_{i}).\]
We show that $\langle s \cup t, x'\rangle$ forces one of the three possible conclusions in the statement of the lemma.

By possibly extending $t$, we may assume that $t$ contains at least one element that is greater than $a$. Let $l_0:=\max(t \cap a)$ and $l^{*}:= \min (t \setminus a)$. Furthermore, let $m^{*}:=\max (x_2 \cap l^{*})$. Hence, $l_0$ and $l^{*}$ are consecutive elements of $t$ and $l_0 \leq m^{*} < l^{*}$ and $l_0 < a \leq l^{*}$. We distinguish between two cases:

\paragraph{Case I.} Assume $l_0 \leq m^{*} \leq a \leq l^{*}$.

If  $l^{*}\in [k_{i-1}, k_{i})$, we are done, since this means that $\langle s \cup t, x'\rangle \Vdash_{\mathbb{M}_{\mathcal{U}}} l^{*} \in \text{range}(\underaccent{\dot}{\eta}) \cap [k_{i-1}, k_{i})$. Hence, assume $l^{*} \notin [k_{i-1}, k_{i})$, i.e., $l^{*} \notin N_i$. Note that $l^{*}$ witnesses that
\[N \models \exists l: l=\min(x_2 \setminus a).\]
Hence, by (vi), we have that $l^{*} \in N_{i+1}$ and thus $l^{*} \in [k_{i}, k_{i+1})$.

\paragraph{Case II.} Assume $l_0 < a < m^{*} < l^{*}$.

Let $t':=t\cap a$, i.e., $l_0:=\max t'$, and let $i^{*} \in \omega \setminus \{0\}$ be such that $l_0 \in [k_{i^{*}-1}, k_{i^{*}})$, i.e., $l_0$ first appears in $N_{i^{*}}$. If $i^{*}=i$, we are again done, hence assume that $a \notin N_{i^{*}}$. We will show that $i^{*}=i-1$.

Let $j \in \omega$ be such that $l^{*}$ is the $j$'th elements of $s \cup t$ in increasing order. By Lemmata \ref{lemma1} and \ref{lemma2}, there is $C_{t' \cup \{l^{*}\}} \subseteq (l^{*})^{+}$ such that
\[\langle s \cup t' \cup \{l^{*}\}, x_2 \setminus (l^{*})^{+}\rangle \Vdash_{\mathbb{M}_{\mathcal{U}}} \undertilde{C}\cap (l^{*})^{+}= C_{t' \cup \{l^{*}\}}.\]

Set $D^{*}:= C_{t' \cup \{l^{*}\}} \cap (l_0, m^{*})$. Since 
\[\langle s \cup t' \cup \{l^{*}\}, x_2 \setminus (l^{*})^{+}\rangle \geq_{\mathbb{M}_{\mathcal{U}}} \langle s \cup t, x'\rangle,\]
and since $l_0 < a < m^{*}$ by assumption, we must have $a \in D^{*}$. Furthermore, note that $D^{*}\subseteq C_{t' \cup \{l^{*}\}} \cap (l_0, l^{*}]$ and thus $|D^{*}|=:\gamma \leq 2^{j}$.

Now, $m^{*}$, $l^{*}$ and $D^{*}$ witness that
\[N \models \exists \langle m,l, D\rangle: \begin{cases}
 	 m,l \in x_2 \setminus {l_0}^{+}, m<l, \text{ and} \\
 	  D \subseteq (l_0, m), \text{ and } \\
 	  |D|=\gamma, \text{ and } \\
 	  \langle s \cup t' \cup \{l\}, x_2 \setminus l^{+}\rangle \Vdash_{\mathbb{M}_{\mathcal{U}}} \undertilde{C} \cap (l_0, m)=D.
 \end{cases}
 \]
 Since $l_0$ is the $(j-1)$'th element of $s \cup t'$, we have $l_0 > 2^{j}$ by Fact \ref{speedfact}.\footnote{Note that the additional requirement in Fact \ref{speedfact} that $\max s$ is already larger than $2^{|s|}$ is needed here, since $l_0$ could be $\max s$.} Hence, since $l_0 \in N_{i^{*}}$, it follows that $\gamma\in N_{i^{*}}$. Thus, all the parameters in the above formula lie in $N_{i^{*}}$, which implies that there exists $\langle m^{\dagger}, l^{\dagger}, D^{\dagger}\rangle \in N_{i^{*}+1}$ satisfying the formula.
 
\begin{claim}
$l^{\dagger}\geq a$	
\end{claim}

Note that the proof of this claim will finish the proof of the Lemma, since $l^{\dagger} \in N_{i^{*}+1}$ by (v) and thus $a \in N_{i^{*}+1}\setminus N_{i^{*}}$.

\noindent{\it Proof of Claim.}
Assume by contradiction that $l^{\dagger} < a$, i.e.,
\[l_0 <m^{\dagger}< l^{\dagger} < a < m^{*} < l^{*}.\]
By Lemma \ref{lemma2}, we have that
\[C_{t' \cup \{l^{\dagger}\}} \cap (m^{\dagger})=C_{t' \cup \{l^{*}\}}\cap (m^{\dagger}).\]
Since $\langle s \cup t' \cup \{l^{\dagger}\}, x_2 \setminus (l^{\dagger})^{+}\rangle \Vdash_{\mathbb{M}_{\mathcal{U}}} \undertilde{C} \cap (l_0, m^{\dagger})=D^{\dagger}$, it follows that $C_{t' \cup \{l^{*}\}}\cap (m^{\dagger})= D^{\dagger}$ and hence $D^{\dagger}=D^{*}\cap (l_0, m^{\dagger})$. However, both $D^{\dagger}$ and $D^{*}$ have size $\gamma$ and thus $D^{*}\subseteq (l_0, m^{\dagger})$, which is a contradiction to the fact that $a \in D^{*}$ and $a > m^{\dagger}$.
This completes the proof of the Claim and Lemma\;\ref{lemma3} as well.
\end{proof}

We now only need one final lemma to finish the proof of the proposition 
and thus of the {\sc Main Theorem}.

\begin{lem}\label{lemmaiso}
Let $I:=\{[k_{i}, k_{i+1}): i \in \omega\}$ be any interval partition of $\omega$ and $E$ and $\mathcal{U}$ non-isomorphic $Q$-points. Then there exist $v \in E$ and $u \in \mathcal{U}$ such that
\[\forall i \in \omega \setminus \{0\}: v \cap [k_{i}, k_{i+1})\neq \emptyset \implies 	\begin{cases}
	u\cap [k_{i-1}, k_{i})= \emptyset, \text{ and}	\\
	u\cap [k_{i}, k_{i+1})= \emptyset, \text{ and}	\\
	u\cap [k_{i+1}, k_{i+2})= \emptyset.
	\end{cases}
\]
\end{lem}

\begin{proof}
Say that a $Q$-point element \textit{selects} from an interval partition if it intersects each interval in exactly one point. Let $v_0 \in E$ and $u_0 \in \mathcal{U}$ be such that they select from $I$. Let $f$ be an order-preserving bijection from $v_0$ to $u_0$, extended to a permutation of $\omega$. Thus, for each $i \in \omega$, $f$ sends the element selected by $v_0$ in $[k_{i}, k_{i+1})$ to the element selected by $u_0$ in $[k_{i}, k_{i+1})$. Since $E$ and $\mathcal{U}$ are non-isomorphic, there exist $v_1 \in [v_0]^\omega \cap E$ and $u_1 \in [u_0]^\omega \cap \mathcal{U}$ such that $u_1 \cap f[v_1] = \emptyset$. Hence, for all $i \in \omega \setminus \{0\}$:
\[v_1 \cap [k_{i}, k_{i+1})\neq \emptyset \implies u_1 \cap [k_{i}, k_{i+1})= \emptyset.\]
Both $E$ and $\mathcal{U}$ contain the set
\[y_{\varepsilon}:=\bigcup_{\substack{{i \in \omega}\\{i \equiv \varepsilon \; (\text{mod}\; 3)}}} [k_{i}, k_{i+1}),\]
each for exactly one $\varepsilon=\varepsilon(E), \varepsilon({\mathcal{U}}) \in 3$. Let $v_2:=v_1 \cap y_{\varepsilon(E)} \in E$ and $u_2:=u_1 \cap y_{\varepsilon(\mathcal{U})} \in \mathcal{U}$. If $\varepsilon(E)=\varepsilon(\mathcal{U})$ then $v_2$ and $u_2$ satisfy the lemma, hence assume without loss of generality that $\varepsilon(E)=0$ and $\varepsilon(\mathcal{U})=1$.

Let $\bar v_0 \in E$ and $\bar u_0 \in \mathcal{U}$ be elements that select from the interval partition
\[\{[k_{i}, k_{i+2}): i \in \omega, i \equiv 0 \;(\text{mod}\;3)\} \cup \{[k_{i}, k_{i+1}): i \in \omega, i \equiv 2 \;(\text{mod}\;3)\}.\]
Again, by considering a permutation of $\omega$ that maps the element selected by $\bar v_0$ in any interval to the element selected by $\bar u_0$ in the same interval, we find $\bar v_1 \in [\bar v_0]^\omega \cap E$ and $\bar u_1 \in [\bar u_0]^\omega \cap \mathcal{U}$ such that $\bar v_1$ and $\bar u_1$ never select from the same interval. Now, clearly, $v_1 \cap \bar v_1 \in E$ and $u_1 \cap \bar u_1 \in \mathcal{U}$ work.
\end{proof}

We can now finish the proof of the proposition and hence of the main theorem: Let $v \in E$, $u \in \mathcal{U}$ be given by the previous lemma for the interval partition $\{[k_{i}, k_{i+1}): i \in \omega\}\cup \{[0, k_0)\}$ constructed in the proof of Lemma~\ref{lemma3}. Let $G \ast H$ be any $\mathbb{M}_{\mathcal{U}}\ast \undertilde{R}$-generic filter containing $\langle \langle s, x_2 \rangle, \undertilde{q}'\rangle$. By Lemma~\ref{lemma3}, we have that in $\mathbf{V}[G \ast H]$, whenever $\text{range}(\undertilde{a}[G \ast H]) \setminus (\max s)^{+}$ intersects one of the intervals $[k_{i}, k_{i+1})$, then the Mathias real $\eta$ intersects $[k_{i}, k_{i+1})$ or one of the adjacent intervals $[k_{i-1}, k_{i})$ or $[k_{i+1}, k_{i+2})$. Since $\text{range}(\eta)$ is almost contained in $u$, the same is true for $u$ in place of $\eta$ above some $n \geq (\max s)^{+}$. Hence, $\text{range}(\undertilde{a}[G \ast H]) \setminus n$ is disjoint from $v$.
\end{proof}

\let\OLDthebibliography\thebibliography
\renewcommand\thebibliography[1]{
  \OLDthebibliography{#1}
  \setlength{\parskip}{0pt}
  \setlength{\itemsep}{5pt plus 0.3ex}
}

\end{document}